\documentclass{amsart}

\usepackage{amsmath, amssymb,amscd,  latexsym, graphicx, mathrsfs, enumerate}

\usepackage{mathtools} 

\usepackage[all]{xy}

\usepackage{color}
\newcommand{\nc}{\newcommand}

\DeclarePairedDelimiter\floor{\lfloor}{\rfloor}

\newtheorem{theorem}{Theorem}[section]
\newtheorem{prop}[theorem]{Proposition}
\newtheorem{importnota}[theorem]{Important Notation}
\newtheorem{prblm}[theorem]{Problem}
\newtheorem{notation}[theorem]{Notation}
\newtheorem{caution}[theorem]{Caution}
\newtheorem{remark}[theorem]{Remark}
\newtheorem{lemma}[theorem]{Lemma}
\newtheorem{construction}[theorem]{Construction}
\newtheorem{corollary}[theorem]{Corollary}
\newtheorem{example}[theorem]{Example}
\newtheorem{conclusion}[theorem]{Conclusion}
\newtheorem{triviality}[theorem]{Triviality}
\newtheorem{proto}[theorem]{Prototype Quasifibration}
\newtheorem{cauex}[theorem]{Cautionary Example}
\newtheorem{propositiondef}[theorem]{Proposition-Definition}
\newtheorem{subth}{Nuisance}[theorem]
\newtheorem{ssubth}{ }[subth]
\newtheorem{conjecture}[theorem]{Conjecture}
\newtheorem{sidest}[theorem]{Side Story}
\newtheorem{miniexample}[theorem]{Example}
\theoremstyle{definition}
\newtheorem{defin}[theorem]{Definition}

\nc\tri[1]{\begin{triviality}}
\nc\side[1]{\begin{sidest}}
\nc\conj[1]{\begin{conjecture}}
\nc\prodef[1]{\begin{propositiondef}}
\nc\prt[1]{\begin{proto}}
\nc\lem[1]{\begin{lemma}}
\nc\sblm[1]{\begin{sublemma}}
\nc\pro[1]{\begin{prop}}
\nc\thm[1]{\begin{theorem}}
\nc\cor[1]{\begin{corollary}}
\nc\dfn[1]{\begin{defin}}
\nc\sthm[1]{\begin{subth}}
\nc\exm[1]{\begin{example}}
\nc\miniexm[1]{\begin{miniexample}}
\nc\plm[1]{\begin{prblm}}
\nc\rmk[1]{\begin{remark}}
\nc\subrmk[1]{\begin{subremark}}
\nc\ntn[1]{\begin{notation}}
\nc\cau[1]{\begin{caution}}
\nc\imn[1]{\begin{importnota}}
\nc\cax[1]{\begin{cauex}}
\nc\con[1]{\begin{construction}}
\nc\ssthm[1]{\begin{ssubth}}
\nc\cnc[1]{\begin{conclusion}}
\nc\elem{\end{lemma}}
\nc\esblm{\end{sublemma}}
\nc\eside{\end{sidest}}
\nc\econj{\end{conjecture}}
\nc\eprodef{\end{propositiondef}}
\nc\eprt{\end{proto}}
\nc\ethm{\end{theorem}}
\nc\ecor{\end{corollary}}
\nc\edfn{\end{defin}}
\nc\esthm{\end{subth}}
\nc\epro{\end{prop}}
\nc\etri{\end{triviality}}
\nc\eexm{\end{example}}
\nc\eminiexm{\end{miniexample}}
\nc\ermk{\end{remark}}
\nc\subermk{\end{subremark}}
\nc\eplm{\end{prblm}}
\nc\ecau{\end{caution}}
\nc\ecax{\end{cauex}}
\nc\eimn{\end{importnota}}
\nc\entn{\end{notation}}
\nc\econ{\end{construction}}
\nc\ecnc{\end{conclusion}}
\nc\essthm{\end{ssubth}}

\newcommand{\lra}{\longrightarrow}

\newcommand{\inj}{\hookrightarrow}

\newcommand{\C}{\mathbb{C}}

\newcommand{\Q}{\mathbb{Q}}

\newcommand{\Z}{\mathbb{Z}}

\newcommand{\X}{\mathfrak{X}}

\newcommand{\w}{\varpi}

\newcommand{\ds}{\displaystyle}

\newcommand{\tzz}{\widetilde{\mathcal{Z}}}
\newcommand{\zz}{\mathcal{Z}}
\newcommand{\tww}{\widetilde{\mathcal{W}}}
\newcommand{\ww}{\mathcal{W}}
\newcommand{\txx}{\widetilde{\mathcal{X}}}
\newcommand{\xx}{\mathcal{X}}
\newcommand{\tyy}{\widetilde{\mathcal{Y}}}
\newcommand{\yy}{\mathcal{Y}}
\newcommand{\tee}{\widetilde{\mathcal{E}}}

\newcommand{\bes}{\begin{equation*}}
\newcommand{\ees}{\end{equation*}}

\newcommand\Prym{\mbox{Prym}}

\def\C{\widetilde{C}}

\def\a{\alpha}
\def\b{\beta}
\def\X{\mathcal X}
\def\Y{\mathcal Y}
\def\<{\langle}
\def\>{\rangle}

\title[On Prym varieties for the coverings of some singular plane curves]{On Prym varieties for the coverings of some singular plane curves}

\author{Lubjana Beshaj \and Takuya Yamauchi}

\keywords{Algebraic curves, Jacobians, and Prym varieties}
\thanks{The second author is partially supported by JSPS Grant-in-Aid for Scientific Research (C) No.15K04787.}
\subjclass[2010]{14H40}

\address{Lubjana Beshaj \\
Department of Mathematics \\
The University of Texas at Austin\\
2515 Speedway, Austin, TX 78712-1202, USA }
\email{beshaj@math.utexas.edu}

\address{Takuya Yamauchi \\ 
Mathematical Inst. Tohoku Univ.\\
 6-3,Aoba, Aramaki, Aoba-Ku, Sendai 980-8578, JAPAN}
\email{yamauchi@math.tohoku.ac.jp}

\begin{document}

\begin{abstract} 
Let $k$ be a field of characteristic zero containing a primitive $n$-th root of unity.   
Let $C^0_n$ be a singular plane  curve of degree $n$ over $k$ admitting an  order $n$ automorphism,  $n$ nodes as the singularities, and
$C_n$ be its normalization.
  
In this paper we study the factors of Prym variety $\Prym(\C_n/C_n)$ associated to the double cover $\C_n$ of $C_n$ exactly ramified at the points obtained by the  blow-up of the  singularities. We provide explicit models of some algebraic curves related to the  construction of  $\Prym (\C_n/C_n)$ as a Prym variety and determine the interesting simple factors other than elliptic curves or hyperelliptic curves with small genus which come up in $J_n$ so that  the endomorphism rings contains the totally real field $\Q(\zeta_n+\zeta^{-1}_n)$. 
\end{abstract}

\maketitle

\section{Introduction}\label{intro}
Let $k$ be a field of characteristic zero that contains   a primitive $n$-th root  of unity $\zeta$ for a positive integer $n$.   
Let $C^0_n$ be a plane curve of degree $n$  in $\mathbb{P}^2_k$  with an automorphism of order $n$ over $k$.   Assume that $C^0_n$ admits $n$ nodes $R_0:=\{P_1,\ldots,P_n\}$ defined over $k$ as the singularities.  Then the geometric genus of $C^0_n$ is 
\[ g(C^0_n)=\frac{(n-1)(n-2)}{2}-n=\frac{n^2-5n+2}{2}.\]
Let $C_n$ be the normalization of $C^0_n$ obtained by blowing up along $R_0$.  If we write the blow-up for $\pi_0:C_n\lra C^0_n$ and put $R=\pi^{-1}_0(R_0)$, then $R$ consists of $2n$ points. 

We consider the double cover $\pi:\C_n\lra C_n$ ramified exactly along $R$. 
Then by genus formula we see that 
\[g(\C_n)=1+(2g(C_n)-2)+\frac{1}{2}|R|=n^2-4n+1\]
since $g(C_n)=g(C^0_n)$ where $g(X)$ stands for the geometric genus of a projective irreducible curve $X$. Let $\iota:\C_n\lra \C_n$ be the involution associated to the covering $\pi$.  Then the endomorphism $1-\iota$ on ${\rm Jac}(\C_n)$ annihilates ${\rm Jac}(C_n)$.  Let $J_n=\Prym (\C_n/C_n)$ be the Prym variety associated to the covering $\pi$ which is defined by 
\[J_n:=(1-\iota){\rm Jac}(\C_n).\]
which is of dimension $g(\C_n)-g(C_n)=\ds\frac{n(n-3)}{2}$. 
It is well-known that $J_n$ is a principal polarized abelian variety (cf. \cite{Mpv}).

In this paper we focus on  the decomposition of $J_n$ over a suitable (finite) extension $L$ of $k$ for  some families of plane singular curves with the same properties as $C^0_n$. 
The motivations would be related to Section 2 of \cite{GY} and the splitting of Jacobians (cf. \cite{beshaj, MR3508311, dec}). 
The points why we study such a curve $C^0_n$ are as follows: 
\begin{enumerate}
\item A smooth plane curve of degree $d$ tends to have high genus in its degree and therefore the decomposition 
of the Jacobian might be hard to analyze. However we can reduce the genus if we consider 
curves of the same degree $d$ with singularities;
\item one can easily find (possibly singular) plane curves which have many automorphisms; 
\item a smooth low genus curve might not have many automorphisms, but a singular curve with the same geometric genus 
could have it. 
\end{enumerate}  
In this vein we explain our main result. Let $C_n$ be a normalization of the singular plane curve in $\mathbb{P}^2_k$ defined by 
$$f_n(x,y,z):=x^n+z^n+(n-2)y^n-nxzy^{n-2}$$
$$+\sum_{i=2}^{\floor{\frac n 2}}   a_i  \{(xz)^iy^{n-2i}-ixzy^{n-2}+(i-1)y^n\}=0$$
whose singularities are exactly given by $n$ nodes $P_i=(\zeta^i:1:\zeta^{-i}),0\le i\le n-1$ for generic parameter $a_i$. 
Let $\C_n$ be the double cover of $C_n$ defined by $\w^2=y^2-xz$ and $f_n(x,y,z)=0$ in $\mathbb{P}^3_k$. 
The curve $C_n$ admits the automorphisms 
$$\alpha:(x:y:z)\mapsto (\zeta x:y:\zeta^{-1}z),\ \beta:(x:y:z)\mapsto (z:y:x)$$
and they are naturally extended to $\C_n$ which commute with the double cover $\C_n\lra C_n$. To abbreviate the notation we write  
$$J_{n,\tau}={\rm Prym}(\C_n/C_n,\tau):={\rm Prym}((\C_n/\<\tau\>)/(C_n/\<\tau\>))$$
for $\tau\in \<\alpha,\beta\>\subset {\rm Aut}_k(C_n)$.   
Then we have a structure theorem in the decomposition of $J_n:={\rm Prym}(\C_n/C_n)$: 
\begin{theorem}\label{m1}$($Theorem \ref{main}$)$ Let $n=2^km\ge 4$ for $k\ge 0$ and odd $m\ge 1$. Assume that $C_n$ is smooth outside $n$-nodes. Then 
\begin{enumerate}
\item  If $n$ is odd (hence $k=0$), then 
$$J_n\stackrel{k}{\sim}J_{n,\alpha}\times  J_{n,\beta}^2$$
so that ${\rm End}_k(J_{n,\beta})\otimes_\Z\Q \supset  \Q(\zeta_n+\zeta^{-1}_n)$.

\item  If $n\ge 6$ is even, then 
$$J_n\stackrel{k}{\sim}J_{n,\alpha}\times  J_{n,\beta}\times J_{n,\beta\alpha^{\frac{n}{m}}}$$
where the latter two Prym varieties are an abelian variety over $k$ whose endomorphism ring contains $\Q(\zeta_n+\zeta^{-1}_n)$ 
respectively. 

\item  If $n=4$, then  $$J_4\stackrel{k}{\sim}J_{4,\beta}\times  J_{4,\beta\alpha^2}$$
where the both  factors in the right hand side are elliptic curves over $k$.

\end{enumerate} 

\end{theorem}

This paper is organized as follows.  In section 2 we introduce our family of singular plane curves and discuss  their automorphism groups. Then, in  Section 3 
we compute an explicit basis of the curves $\C_n$ and $\C_n$. In section 4 we give a proof of Theorem \ref{m1} according to a main idea invented in \cite{GY}. 
In section 5 we perform explicit computations for  $4\le n \le 8$ and give an explicit decomposition of the corresponding 
Prym varieties. We also try to realize some of Prym factos as a Jacobian of a curve as long as the computation carries out without any 
complexity. 

Throughout this paper $k$ is a field of characteristic zero containing all $n$-th roots of unity with infinite cardinality and  
we denote by $D_n$ the dihedral group of order $2n$, and $\C$ the normalization for a projective algebraic  curve $C$. 

\medskip

\textbf{Acknowledgments.} We would like to thank Professor Tony Shaska for helpful comments and hearty encouragements.

\section{Families of singular curves with automorphisms}
Let $n\ge 4$. 
Let $\zeta:=\zeta_n$ be a primitive $n$-th root of unity and $k$ a characteristic zero field such that $k( \zeta_n) = k$. 
In this paper we consider the following singular plane curve $C^0_n$ in $\mathbb{P}^2$ defined over $k$ 
and given by the equation $f_n(x,y,z)=0$, where 
\begin{eqnarray}\label{eq-1}
 f_n(x,y,z)&:=&x^n+z^n+(n-2)y^n-nxzy^{n-2}\\
&&+\sum_{i=2}^{\floor{\frac n 2}}   a_i  \left\{(xz)^iy^{n-2i}-ixzy^{n-2}+(i-1)y^n \right\} \nonumber
 \end{eqnarray}
and $\floor{\frac n 2}$ stands for the maximal integer less than or equal to $\ds\frac{n}{2}$ and $a_i$'s are parameters.  
The curve $C^0_n$ has $n$ nodes 
$P_i:=(\zeta^i:1:\zeta^{-i}),\ i=0,\ldots,n-1$ and admits two automorphisms:
\begin{equation}\label{auto1}
\begin{split}
\alpha& :(x:y:z)\mapsto (\zeta x:y:\zeta^{-1}z) \\
\beta& :(x:y:z)\mapsto (z:y:x),  
\end{split}
\end{equation}
defined over $k$ since $\zeta \in k$. 
Throughout this paper we assume that $C^0_n$ is smooth at any point other than $P_i,\ i=0,\ldots,n-1$. 
One can check this condition would not be so restrictive. In fact  
when $a_i=0$ for all $2\le i\le \floor{\frac n 2}$ and $k=\mathbb{C}$, we see 
$$C^0_n:x^n+z^n+(n-2)y^n-nxzy^{n-2}=0$$
which is smooth at any point other than the $n$-nodes. Hence by continuity our family $C^0_n$ gives rise to 
a family whose generic member is smooth other than the $n$-nodes. 

\begin{remark}
In the case of $n=5$ the genus of $C^0_5$ is one. On the other hand any projective smooth curve of genus one does not admit 
an automorphism of order $n=5$. However $C^0_5$ does  because of the singularities. 
\end{remark}

Next we consider the double covering in $\mathbb{P}^3$ defined by:
\begin{equation}\label{covering}
\C^0_n:
\left\{
\begin{array}{c}
w^2=y^2-xz \\
f_n(x,y,z)=0.
\end{array}\right.
\end{equation}
This curve admits three automorphism 
\begin{equation}\label{auto2}
\begin{split}
\alpha& :(x:y:z:w)\mapsto (\zeta x:y:\zeta^{-1}z:w) \\
\beta& :(x:y:z:w)\mapsto (z:y:x:w) \\
\gamma& :(x:y:z:w)\mapsto (x:y:z:-w). 
\end{split}
\end{equation}
Here we use the same notation for $\alpha,\beta$ because they are naturally coming from (\ref{auto1}). 
Let $\C_n$ be the normalization of $\C^0_n$ and the double cover extends naturally to 
a double cover $\C_n\lra C_n$ which inherits the automorphisms (\ref{auto2}). 
Put $J_n=\Prym(\C_n/C_n)$ as before. To study a decomposition of $J_n$ and its factor we make use of the quotient curves by 
automorphisms.

\subsection{Automorphisms}

The decomposition of Jacobians of such curves will be realized by investigating the automorphism group of each curve.  Hence, we would like to study a specific subgroup in  
${\rm Aut}_k(C_n^0)$ and ${\rm Aut}_k(\C_n^0)$ respectively.

\begin{lemma} 
Let $C_n^0$ and $\C_n^0$ be as above.   Then $D_{2n} \inj {\rm Aut}_k(C_n^0)$ and $\Z/2\Z \times D_{2n} \inj {\rm Aut}_k(\C_n^0)$. 
\end{lemma}

\proof
It is obvious that $\a$ and $\b$ have orders $n$ and 2 respectively.    Let $G$ be the group generated by $\a, \b$.  Obviously, $G \inj {\rm Aut}_k (C_n^0)$. Since $\b \a \b = \a^{-1}$, then $G $ is isomorphic to the dihedral group of order $2n$. 

To prove the second part of the theorem we let    $G := \langle \a,   \b, \gamma \rangle$, where $\a, \b, \gamma$ are as in (\ref{auto2}). 
Obviously, $\a$ has oder $n$ and $\b$ and $\gamma$ are involutions. 
Since $\a$ and $\b$ generate a dihedral group and $\gamma $ commutes with $\a$ and $\b$, then they generate a group of order $4n$ which is isomorphic to $\Z/2\Z \times D_{2n}$. 
\qed

\subsection{Decomposing the Jacobian by group partitions}
Let $\X$ be a genus $g$  algebraic curve defined over $k$ with automorphism group $G = {\rm Aut}_k (\X)$. 

The following result is from \cite[Theorem B]{KR}:
\begin{prop}
Let $H \subset G$ such that $H = H_1 \cup \dots \cup H_t$ where the subgroups $H_i \subset H$ satisfy $H_i \cap H_j = \{ 1\}$ for all $i\neq
j$.  Then, we have the isogeny relation 
$${\rm Jac}(\X )^{t-1} \times {\rm Jac}(\X / H)^{|H|}\,  \stackrel{k}{\sim} \, {\rm Jac}(\X / H_1)^{| H_1 |} \times \cdots \times {\rm Jac}(\X / H_t)^{| H_t | }$$
\end{prop}
As an easy consequence we have 
\begin{corollary}\label{easy}Keep the notation as above. Assume that $H\simeq \Z/2\Z\times \Z/2\Z$ and write $H_0=\{e\}$ for the 
trivial subgroup and $H_i\simeq \Z/2\Z,\ 1\le i\le 3$ for other three subgroups. Then 
 $${\rm Jac}(\X )\times {\rm Jac}(\X/H)^2\stackrel{k}{\sim}  {\rm Jac}(\X/H_1)\times {\rm Jac}(\X/H_2)\times {\rm Jac}(\X/H_3).$$
\end{corollary}

\section{Holomorphic differential forms}
In this section we study holomorphic differential forms on $C_n$ and $\C_n$ respectively which will be used to 
decompose the Prym varieties in question. We work on an affine singular model $\widehat{f}_n:=f_n(x,1,z)=0$ for $C^0_n$. Put 
$\widehat{f}_{n,x}:=\ds\frac{\partial \widehat{f}_n}{\partial x}$ and $\widehat{f}_{n,z}:=\ds\frac{\partial \widehat{f}_n}{\partial z}$.
\begin{theorem}\label{dif}The followings hold:
\begin{enumerate}
\item The elements 
$$\omega_{r,s}:=\frac{(xz-1)x^rz^sdx}{\widehat{f}_{n,z}},\ 0\le r,s,r+s\le n-5$$
and 
$$\theta_i=(x^i-z^{n-i})\frac{dx}{\widehat{f}_{n,z}},\ 3\le i \le n-3$$
make up a basis of $H^0(C_n,\Omega^1)\simeq H^0(C^0_n,\Omega^1)$.  
\item The elements 
 $$\widetilde{\omega}_{r,s}:=\frac{wx^rz^sdx}{\widehat{f}_{n,z}},\ 0\le r,s,r+s\le n-4$$
and 
$$\widetilde{\theta}_i=(x^i-z^{n-i})\frac{dx}{w\widehat{f}_{n,z}},\ 2\le i \le n-2$$ with a basis of $H^0(C_n,\Omega^1)$ as above 
make up a basis of $H^0(\C_n,\Omega^1)\simeq H^0(\C^0_n,\Omega^1)$. In particular 
$$\<\widetilde{\omega}_{r,s},\widetilde{\theta}_i\ |\ 
\tiny{\begin{array}{c}
0\le r,s,r+s\le n-4\\
 2\le i \le n-2
\end{array}}
\>=H^0(\C_n,\Omega^1)^{\gamma^\ast=-1}\simeq H^0(J_n,\Omega^1_{J_n}).$$  
Here the superscript $\gamma^\ast=-1$ means the maximal subspace so that the pullback $\gamma^\ast$ of $\gamma$ acts 
as the multiplication by $-1$. 
\end{enumerate}
\end{theorem}
\begin{proof}
Put $\omega=\ds\frac{dx}{\widehat{f}_{n,z}}$. In what follows we compute the divisor of $\omega$. We will take a careful analysis around 
the singularities on $C^0_n$. 

Let $P$ be a non-singular point on the affine model $C^0_n$. 
Then ${\rm ord}_P(\omega)=0$ since $\ds\frac{dx}{\widehat{f}_{n,z}}=-\ds\frac{dz}{\widehat{f}_{n,x}}$ and $(\widehat{f}_{n,z}(P),
\widehat{f}_{n,x}(P))\not=(0,0)$. Let $R$ be the strict transform of $n$-nodes $P_0,\ldots,P_{n-1}$ for the normalization $\pi:C_n\lra C^0_n$. It 
consists of $2n$ points $\{Q_i,\beta(Q_i)\ |\ i=0,\ldots,n-1\}$ which will be specified later. We first consider 
the point $P_0=(1:1:1)$. Put $x_1=x+1$ and $z_1=z+1$. 
Consider the curve defined by 
$$f_n(x_1+1,1,z_1+1)=0$$
and blow up it along $P'_0:=(0,0)$ in $(x_1,z_1)$-plane $\mathbb{A}^2_{(x_1,z_1)}$. 
We have an affine (isomorphic) model of a desingularization of $C^0_n$ at $P_0$ by gluing 
$$U_{x_1}:\left\{\begin{array}{c}
z_1=x_1s \\
f(x_1,s)=0
\end{array}
\right. \subset \mathbb{A}^3_{(x_1,z_1,s)}\ {\rm and}\  
U_{z_1}:\left\{\begin{array}{c}
x_1=z_1t \\
g(z_1,t)=0
\end{array}
\right. \subset \mathbb{A}^3_{(x_1,z_1,t)}  
$$
with the relation $st=1$. Here $f$ is defined by the relation 
$f_n(1+x_1,1,1+x_1s)=x^2_1f(x_1,s)$. It is easy to see that $f(x_1,s)$ takes the following form 
$$f(x_1,s)=p_0+p_1s+p_0s^2+x_1q(x_1,s),\ p_0,p_1\in k[a_i]\setminus\{0\},\ q(x_1,s)\in k[x_1,s].$$
By symmetry we have $g(x_1,s)=f(x_1,s)$. Then we may put $Q_0=(x_1,z_1,s)=(0,0,\alpha_1)$ where $\alpha_1$ is 
a root of $p_0+p_1s+p_0s^2=0$ with another root $\alpha_2$ and then we have $\beta(Q_0)=(0,0,\alpha_2)$. 
Since $\omega=\ds\frac{dx_1}{x_1(p_1+2p_0s+x_1\frac{\partial q}{\partial s}(x_1,s))}$, if we choose $x_1$ as a local parameter at $Q_0$, then we see that 
${\rm ord}_{Q_0}(\omega)=-1$. Therefore for any $r,s\ge 0$, the differential form $\omega_{r,s}=(xz-1)x^rz^s\omega$ is 
holomorphic at $Q_0$ and $\beta(Q_0)$ because of the factor $xz-1$. By a similar argument one can treat other points in $R$ though 
the symmetry for $\beta$ would be violated.  
So we omit details. What we remain to do is to check the order at infinity points. 

Let $x=\ds\frac{1}{u}$ with a local parameter $u$ at an infinity point. Note that there are exactly $n$ infinity points and they are 
corresponding to $(x:0:z)$ satisfying $f_n(x,0,z)=0$. It is easy to check that ${\rm ord}_u(z)=-1$ since 
$f_n(x,0,z)=x^n+z^n+({\rm lower\ terms})=0$. One can also check ${\rm ord}_u(\omega)=n-3$ by direct computation. 
Therefore we have 
$${\rm ord}_u(\omega_{r,s})=-2-r-s+n-3=n-5-(r+s)\ge 0$$
by the assumption on $r,s$.  

For $\theta_i$, the factor $x^i-z^{n-i}$ vanishes at all $P_j,\ 0\le j\le n-1$ (and also at all points in $R$). Therefore 
we may consider the order at infinity points. The details are omitted. 
Summing up we have $$\ds\frac{(n-4)(n-3)}{2}+n-5=\frac{n^2-5n+2}{2}=g(C_n)$$ elements which give a basis of $H^0(C_n,\Omega^1)$. 

For the second claim the computation would be similar. So we only give key points. 
Let $u$ be a local parameter of a point $P$ in $R$. Since the double cover $\C_n\lra C_n,w^2=1-xz$ is ramified at $P$, one has 
${\rm ord}_u(x)=2$ and ${\rm ord}_u(w)=1$. Therefore we have 
$${\rm ord}_u(w\frac{dx}{f_z})=1+1-2=0,\ {\rm ord}_u(\frac{dx}{w})=1-1=0.$$
These imply the holomorphy for elements $\widetilde{\omega}_{r,s},\widetilde{\theta}_i$ in the claim with the same analysis at infinite points. 
We have  $$\ds\frac{(n-3)(n-2)}{2}+n-3=\frac{n^2-3n}{2}=g(\C_n)-g(C_n)$$ elements  so that $\gamma^\ast$ acts as the multiplication by $-1$. 
Hence they give a basis of $H^0(\C_n,\Omega^1)^{\gamma^\ast=-1}$. 
\end{proof}

\section{Decomposition of Jacobians $J_n$.}  
In this section we determine the subcovers of $C_n^0$ and $\C_n^0$ fixed by their automorphisms and their smooth models. 
We start to work with $C^0_n$ to compute the defining equations of quotient curves but consider a smooth model at the same time. 
Therefore we often state our results for a smooth model of a singular curve appear here. 

In what follows we consider the following quotient curves. 
for any $n\ge 4$ we put  
$$\X_n:=C_n/\<\alpha\>,\ \widetilde\X_n:=\C_n/\<\alpha\>,\ \yy_n:=C_n/\<\beta\>,\ \tyy_n:=\C_n/\<\beta\>$$
For even $n$ we write $n=2^km\ge 6$ with $k\ge 1$ and odd $m$. Put 
$$ \zz_n:=C_n/\<\beta\alpha^{\frac{n}{m}}\>,\ \tzz_n:=\C_n/\<\beta\alpha^{\frac{n}{m}}\>.$$
Finally for $n=4$, put 
$$ \ww_4:=C_4/\<\beta\alpha^{2}\>,\ \tww_4:=\C_4/\<\beta\alpha^{2}\>.$$
It will turn out that $\widetilde\X_n$ is a hyperelliptic curve with the hyperelliptic involution $\delta$. Then we put 
$$\tee_n:=\widetilde\X_n/\<\delta\gamma\>=\C_n/\<\a,\delta\gamma\>.$$
Note that the automorphisms other than $\alpha$ are all involutions in each case.  
We also introduce the singular version $\X^0_n, \widetilde\X^0_n,\ldots$ etc in the same manner as above. 
Hence we have obtained the following figures:

\begin{figure}[hb]\scalebox{.85}
{
\xymatrix{
& & \C_n  \ar@{->}[dl]_n \ar@{->}[dr]^2   \ar@{->}[dd]^2 &&&& \\
\tee_n:=\C_n/\<\a,\delta\gamma\> & \ar@{->}[l]_2 \widetilde\X_n:= \C_n/\<\a\> \ar@{->}[dd]^2 &  & \widetilde\Y_n:=\C_n/\<\b\>  \ar@{->}[dd]^2 
&&&& \\
&  & C_n   \ar@{->}[dl]_n \ar@{->}[dr]^2  &  &&&&&\\
&\X_n:=C_n/\<\a\> &      & \Y_n:= C_n/\<\b\>  &&&&
}
}
\caption{The diagram of maps}
\end{figure}

\begin{figure}[hb]\scalebox{.85}
{
\xymatrix{
n=2^km\ge 6\ {\rm with}\ {\rm odd}\ m & \C_n  \ar@{->}[dl]_2 \ar@{->}[dr]^2   \ar@{->}[dd]^2 & n=4   \\
  \tzz_n:= \C_n/\<\b\alpha^{\frac{n}{m}}\> \ar@{->}[dd]^2 &  & \tww_4:=\C_4/\<\b\alpha^{2}\>  
\ar@{->}[dd]_2   \\
  & C_n   \ar@{->}[dl]_2 \ar@{->}[dr]^2  &  &  \\
\zz_n:=C_n/\<\b\alpha^{\frac{n}{m}}\> &      & \ww_4:= C_4/\<\b\alpha^{2}\>  
}
}
\caption{The diagram of maps}
\end{figure}
 
At first we start with the subcovers fixed by $\a$. 

\begin{prop}\label{x}
\noindent
i) The curve $\X_n^0$ has the following affine model
\[ \X_n^0 : \quad s^2 + s \left(   n (1-t) - 2 + \sum_{i=2}^{\floor{\frac n 2}} a_i \, ( t^i- it +i -1  )  \right)   + t^n =0.   \]

\medskip
\noindent
ii) The smooth model $\X_n$ of $\X_n^0$ is a hyperelliptic (or elliptic) curve of generic genus $g= \floor*{\frac {n-3} 2  }$ defined by 
\[ \X_n : \quad u^2=f^0_n(t)     \]
where $f^0_n(t)=\ds\frac{1}{(t-1)^2}\Big(\frac{A^2}{4}-t^n\Big)\in k[t]$ and $A:=  n (1-t) - 2 + \ds\sum_{i=2}^{\floor{\frac n 2}} a_i \, ( t^i- it +i -1  ) $. 

\medskip
\noindent
iii) The smooth model $\widetilde\X_n$ of $\widetilde\X_n^0$ is a hyperelliptic curve of generic genus $g=n-3$ defined by 
\[ \widetilde\X_n : \quad u^2=f^0_n(1-w^2)     \] 

\medskip
\noindent
iv) The smooth model $\widetilde{\mathcal{E}}_n$ of $\widetilde{\mathcal{E}}_n^0$ is a hyperelliptic curve of generic genus $g=\floor*{\frac {n-2} 2  }$ defined by 
\[ \widetilde{\mathcal{E}}_n : \quad U^2=Wf^0_n(1-W)     \]  
\end{prop}

\proof We start with the curve $C_n^0$ and its automorphism $\a$. 
It is obvious from Eq. (\ref{eq-1}) that  $\alpha$ fixes $s=x^n$ and $t= xz$. Thus, the equation of $f_n(x,y,z)=0$ gives  
\[ x^{2n} + t^n + x^n \left( y^n (n-2) - n y^{n-2}t + \sum_{i=2}^{\floor{\frac n 2}} a_i \left( t^i y^{n-2i} - i t \, y^{n-2} + (i-1)y^n \right)   \right) =0 \]
Putting $y=1$ we have an affine model 
\[ \X_n^0 : \quad  s^2 + s \, \left(n-2 - nt + \sum_{i=2}^{\floor{\frac n 2}} a_i \left( t^i  - i t \,+ i-1 \right)  \right) + t^n =0\]
in the first claim. 
%
%
%
%
Denote by \[A= \left(   n (1-t) - 2 + \sum_{i=2}^{\floor{\frac n 2}} a_i \, ( t^i- it +i -1  )  \right) .\]  
By making the substitution   $u'= s+\ds\frac A 2$ we get $u'^2 = \ds\frac {A^2} 4 - t^n$. 
Since $\ds\frac{d^i}{dt^i}\Big(\ds\frac {A^2} 4 - t^n\big)\Big|_{t=1}=0$ for $i=0,1$, the polynomial $ \ds\frac {A^2} 4 - t^n$ is 
divisible by $(t-1)^2$.    
Put $$f_n^0(t):=\ds\frac{1}{(t-1)^2}\Big(\frac{A^2}{4}-t^n\Big)$$
as in the statement. 
The map $C_n\lra \X_n$ is unramified everywhere if $n$ is odd while it is ramified exactly at $n$-points $(x:0:z),\ f_n(x,0,z)=0$ with 
the ramification index 2 when $n$ is even. Note that such $n$ points 
are fixed by $\alpha^{\frac{n}{2}}$. It follows from this that the genus of $\X_n$ (resp. $\widetilde\X_n$) is $\floor*{\frac {n-3} 2  }$ 
(resp. $n-3$). 
This means that $f^0_n(t)$ is a separable polynomial of degree $n-2$ in $t$.  
Substitute $u'$ with $u=\ds\frac{u'}{(t-1)^2}$ a generic member $\X_n^0$ defined by   
\[ \X_n^0: \quad  u^2= f_n^0 (t)   \]
is a hyperelliptic curve of genus $\floor*{\frac {n-3} 2  }$.

Since the covering $\widetilde\X_n\lra \X_n $ is given by $w^2=1-xz=1-t$, one has  
\[\widetilde\X_n :   \quad u^2=f^0_n(1-w^2)     \]
with the hyperelliptic involution $\delta:(u, w)\mapsto (-u, w)$.   Then clearly 
\[ \delta \, \gamma:   (u, w)\mapsto (-u,-w)\]
 and this gives rise to a model of $ \widetilde {\mathcal E}_n  $ as below since $\widetilde\X\lra \widetilde {\mathcal E}_n  = \widetilde C_n / \<\alpha, \delta \gamma\>=\widetilde\X/\<\delta \gamma\>$:
\[ \widetilde {\mathcal E}_n   :   \quad U^2=Wf^0_n(1-W),\ (W,U)=(w^2,uw).      \]

\qed

Summing up we have obtained the following commutative diagram: 
\[
\begin{CD}
\widetilde C_n  @> n:1  > {\rm etale} >   \widetilde\X_n 
@>2:1>>\widetilde {\mathcal E}_n,  \\
  @V 2:1 VV @V 2:1 VV @.\\
 C_n  @> n:1 > > \X_n @. 
\end{CD}\ \qquad  
\begin{CD}
n^2-4n+1  @>  >>  n-3
@>>> \floor*{\frac {n-2} 2  }  \\
  @V VV @V VV @.\\
 \frac{n^2-5n+2}{2}  @> >> \floor*{\frac {n-3} 2  }. @. \\  
\end{CD} \\
\] 

\medskip
Next we determine equations for the quotient curves $\Y^0_n$ and $\widetilde\Y^0_n$ and their genii. 
Let $P_n(u,v)$ be the  two variable polynomial over $\Z$ so that $P_n(x+z,xz)=x^n+z^n$. 

\begin{prop}\label{y}
i) The curve $\Y^0_n$ has genus 
$g(\Y^0_n)=\left\{\begin{array}{cc}
\frac{(n-1)(n-5)}{4} & \mbox{$n$ is odd} \\
\frac{(n-2)(n-4)}{4} & \mbox{$n$ is even}
\end{array}\right.
$
 and it has an affine model 
\[ P_n(u,v)+(n-2)-nv+\sum_{i=2}^{\floor*{\frac {n} 2  }}a_i(v^i-iv+i-1)=0  \]
ii) The curve $\widetilde\Y^0_n$ has genus  
$g(\widetilde\Y^0_n)=\left\{\begin{array}{cc}
\frac{(n-1)(n-4)}{2} & \mbox{$n$ is odd} \\
\frac{(n-2)(n-3)}{2} & \mbox{$n$ is even}
\end{array}\right.
$ and it has an affine model 
\[ P_n(u,1-w^2)-2+nw^2+\sum_{i=2}^{\floor*{\frac {n} 2  }}a_i\{(1-w^2)^i-1-iw^2\}=0  \]
\end{prop}

\begin{proof}
Let us introduce the functions 
\[ u = x+z  \quad \textit{and } \quad v= xz \]
which are fixed by $\beta$.  Then we can easily obtain the defining equations in the both claims. 

We now compute the genus only for $\Y^0_n$ because the other case would be done similarly.   
In the case when $n$ is odd, if $P=(x:y:z)$ is fixed by $\beta$, then it has to be 
$$P=(1:0:-1)\ {\rm or}\ P=(x:1:x).$$
In the latter case we see that $f_n(x,1,x)=(x-1)^2h(x)$ for some polynomial $h(x)\in k[x]$ with $h(\zeta^i)\not=0$ for any $0\le i\le n-1$ and deg$_xh(x)=n-2$. 
Under normalization of $C^0_n$, the proper transform of the point $(1:1:1)$ consists of two points interchanged by $\beta$. 
Therefore the number of fixed points of $\beta$ is $1+(n-2)=n-1$. By the Hurwitz formula for the double covering 
$C^0_n\lra \Y^0_n$ defined by $u,v$,  we have the claim. 

In the case when $n$ is even, the fixed points consist of $P=(x:1:x),\ x\not=\pm 1$. In fact we see that 
$f_n(x,1,x)=(x^2-1)^2h_1(x)$ for some polynomial $h_1(x)\in k[x]$ with $h(\zeta^i)\not=0$ for any $0\le i\le n-1$ and deg$_xh_1(x)=n-4$. 
The claim follows by the Hurwitz formula again.   

For $\widetilde\Y^0_n$, in case when $n$ is odd, the point $(1:0:-1:\pm 1)$ is no longer a fixed point of $\beta$ just 
because of the new parameter $w$ and then the number of fixed points are $2(n-2)$. In case when $n$ is even, it is $2(n-4)$. 
The formula follows from these data.  
\end{proof}

Since the geometric genus is stable under the normalization, we have obtained the following commutative diagram: 
\[
\begin{CD}
\widetilde C_n  @> 2:1  > >   \widetilde\Y_n=\widetilde C_n/\<\beta\>,  \\
  @V 2:1 VV @V 2:1 VV @.\\
 C_n  @> 2:1 > > \Y_n=C_n/\<\beta\> @. 
\end{CD}\ \qquad  
\begin{CD}
n^2-4n+1  @>  >>  \left\{\begin{array}{cc}
\frac{(n-1)(n-4)}{2} & \mbox{($n$ is odd)} \\
\frac{(n-2)(n-3)}{2} & \mbox{($n$ is even)}
\end{array}\right.  \\
  @V VV @V VV @.\\
 \frac{n^2-5n+2}{2}  @> >> \left\{\begin{array}{cc}
\frac{(n-1)(n-5)}{4} & \mbox{($n$ is odd)} \\
\frac{(n-2)(n-4)}{4} & \mbox{($n$ is even)}
\end{array}\right. @. \\  
\end{CD} \\
\] 

Let $n=2^k m\ge 6$ with $k\ge 1$ and  $m$ an odd number. 
Before giving the  proof of the main theorem we observe an automorphism 
$$\beta \alpha^{\frac{n}{m}}:(x:y:z)\mapsto (\zeta^{-1}_{2^k}z:y:\zeta_{2^k}x).$$
Clearly it is an involution.   Recall the quotient curves $\tzz_n=\widetilde{C}_n/\<\beta \alpha^{\frac{n}{m}} \>$ and 
$\zz_n=C_n/\<\beta \alpha^{\frac{n}{m}} \>$. Since $2^k|n$, there exists a polynomial $Q_n(X,Y)\in k[X,Y]$ such that 
$$Q_n(x+\zeta^{-1}_{2^k}z,xz)=x^n+z^n.$$
Then we have  the following result.

\begin{prop}\label{z}Let $n$ be as above.  
\begin{enumerate}

\item The curve $\zz_n$ has genus  
$g(\zz_n)=\ds\frac{n^2-6n+4}{4}$ and it has an affine model 
\[ Q_n(u,v)+(n-2)-nv+\sum_{i=2}^{\floor*{\frac {n} 2  }}a_i(v^i-iv+i-1)=0  \]
where $(u,v)=(x+\zeta^{-1}_{2^k}z,xz)$. 
\item The curve $\tzz_n$ has genus 
$g(\tzz_n)=\ds\frac{n^2-5n+2}{2}$
 and it has an affine model 
\[ Q_n(u,1-w^2)-2+nw^2+\sum_{i=2}^{\floor*{\frac {n} 2  }}a_i\{(1-w^2)^i-1-iw^2\}=0.  \]
\end{enumerate}
\end{prop}
\begin{proof}
The number of the fixed points of the involution $\beta\alpha^{\frac{n}{m}}$ on $C_n$ (resp. $\widetilde{C}_n$) is $n$ (resp. $2n$). 
These are given by $(x:1:\zeta_{2^k}x)$ (resp. $(x:1:\zeta_{2^k}x:w)$) satisfying 
$f_n(x,1,\zeta_{2^k}x)=0$ (resp. $f_n(x,1,\zeta_{2^k}x)=0,\ w^2=1+x^2$). 
Note that any fixed point never intersects with $n$-nodes or its proper transform. 
The genus for each curve is obtained from Hurwitz formula. The defining equations can be obtained easily.   
\end{proof}

In the case when $n=4$, let us consider the involution defined by 
$$\alpha^{2}\beta=\beta\alpha^2:(x:y:z)\mapsto (-z:y:-x).$$  By direct calculation we have the following. 
    
\begin{prop}\label{w} The curve $\tww_4$ is of genus one and it is defined by 
$$4 u^2 + u^4 - 4 u^2 w^2 + 2 w^4 +a_2 w^4 =0$$
where $u=x-z$ and $w^2=1-xz$. 
\end{prop}
Let us recall the quotient map $\C_n\lra \tyy_n$. It's pullback induces ${\rm Jac}(\tyy_n)\lra {\rm Jac}(\C_n)\stackrel{1-\gamma^\ast}{\lra} J_n$ which factors through 
${\rm Prym}(\tyy_n/\yy_n)$. Hence we have a natural map ${\rm Prym}(\tyy_n/\yy_n)\lra {\rm Jac}(\C_n)$. 
Similarly we also have a natural map ${\rm Prym}(\tzz_n/\zz_n)\lra {\rm Jac}(\C_n)$. The quotient maps  
$\C_n\lra \txx_n\lra \tee_n$ induce a natural map ${\rm Jac}(\tee_n)\lra J_n$ as well. 
The following fact is necessary to prove the main theorem. 

\begin{prop}\label{di} 
\begin{enumerate}
\item The natural map ${\rm Prym}(\tyy_n/\yy_n)\lra J_n$ induces an identification 
$$H^0({\rm Prym}(\tyy_n/\yy_n),\Omega^1)\simeq \<\widetilde{\omega}_{r,s}-\widetilde{\omega}_{s,r},\widetilde{\theta}_i+
\beta^\ast(\widetilde{\theta}_i)\ |\ 
\tiny{\begin{array}{c}
0\le r<s\le n-4\\
 2\le i \le \floor{\frac{n}{2}}
\end{array}}
\>$$ 
\item 
The natural map ${\rm Prym}(\tzz_n/\zz_n)\lra J_n$ induces an identification 
$$H^0({\rm Prym}(\tzz_n/\zz_n),\Omega^1)\simeq \<\widetilde{\omega}_{r,s}+(\beta \alpha^{\frac{n}{m}})^\ast(\widetilde{\omega}_{s,r}),\widetilde{\theta}_i+
(\beta \alpha^{\frac{n}{m}})^\ast(\widetilde{\theta}_i)\ |\ 
\tiny{\begin{array}{c}
0\le r<s\le n-4\\
 2\le i < \floor{\frac{n}{2}}
\end{array}}
\>$$ 
\item 
The natural map ${\rm Jac}(\tee_n)\lra J_n$ induces an identification 
$$H^0({\rm Jac}(\tee_n),\Omega^1)\simeq \Big\< w(xz-1)^j\frac{dx}{\widehat{f}_{n,z}}\  \Big|\ 
0\le j \le \floor{\frac{n-4}{2}}
\Big\>.$$ 
\end{enumerate}
\end{prop}
\begin{proof}The first claim follows directly from Theorem \ref{dif}. For the second claim, we have  to remark that when $i=\frac{n}{2}$, 
$$\theta_{\frac{n}{2}}+(\beta \alpha^{\frac{n}{m}})^\ast(\theta_{\frac{n}{2}})=
\{(1+\zeta^{\frac{n}{2}}_{2^k})x^{\frac{n}{2}}-(1+\zeta^{-\frac{n}{2}}_{2^k})z^{\frac{n}{2}}\}\frac{dx}{w\widehat{f}_{n,z}}=0$$
since $\zeta^{\frac{n}{2}}_{2^k}=(-1)^m=-1$. Note that $(\beta \alpha^{\frac{n}{m}})^\ast(\ds\frac{dx}{\widehat{f}_{n,z}})=
\frac{dz}{\widehat{f}_{n,x}}=-\frac{dx}{\widehat{f}_{n,z}}$. 

For the third claim, notice that  the natural map $\C_n\lra \txx_n$ induces a natural identification  
$$H^0(\txx_n,\Omega^1)=H^0(\C_n,\Omega^1)^{\alpha^\ast=1}$$
and it is generated by $w^{i}\ds\frac{dw}{u},\ i=0,\ldots, n-4$. Let us remark that $w^i\ds\frac{dw}{u}=w^{i+1}(xz-1)\frac{dx}{\widehat{f}_{n,z}}$ since  
$$u=\frac{s+\frac{A}{2}}{(xz-1)^2}=\frac{x\widehat{f}_{n,x}-z\widehat{f}_{n,z}}{2(xz-1)^2}$$ 
and $2wdw=-zdx-xdz=(x\widehat{f}_{n,x}-z\widehat{f}_{n,z})\ds\frac{dx}{\widehat{f}_{n,z}}$. 
The claim follows from $H^0(\tee_n,\Omega^1)=H^0(\txx_n,\Omega^1)^{\gamma^\ast=-1}=H^0(\txx_n,\Omega^1)^{(\delta\gamma)^\ast=1}$. 
\end{proof}

We now give a structure theorem for a decomposition of the Prym variety $J_n:={\rm Prym}(\widetilde{C}_n/C_n)$: 
\begin{theorem}\label{main} Assume $n\ge 4$.  

\begin{enumerate}
\item  If $n$ is odd, then 
$$J_n\stackrel{k}{\sim}{\rm Jac}(\widetilde{\mathcal{E}}_n)\times {\rm Prym}(\widetilde{\mathcal{Y}}_n/{\mathcal{Y}}_n)^2$$
where the factor ${\rm Prym}(\widetilde{\mathcal{Y}}_n/{\mathcal{Y}}_n)$ is an abelian variety over $k$ whose endomorphism ring contains $\Q(\zeta_n+\zeta^{-1}_n)$.  
Here $\zeta_n\in \mathbb{C}$ is a primitive $n$-th root of unity. 
\item  If $n=2^km\ge 6$ with $k\ge 1$ and  $m$ an odd number, then 
$$J_n\stackrel{k}{\sim}{\rm Jac}(\widetilde{\mathcal{E}}_n)\times {\rm Prym}(\widetilde{\mathcal Y}_n/{\mathcal Y}_n)\times 
{\rm Prym}(\widetilde{\mathcal Z}_n/{\mathcal Z}_n)$$
where the latter two Prym varieties are an abelian variety over $k$ whose endomorphism ring contains $\Q(\zeta_n+\zeta^{-1}_n)$ 
respectively. 
\item  If $n=4$, then 
$$J_4\stackrel{k}{\sim}{\rm Prym}(\widetilde{\mathcal Y}_4/{\mathcal Y}_4)\times 
{\rm Prym}(\widetilde{\mathcal W}_4/{\mathcal W}_4)$$
\end{enumerate}
\end{theorem}

\begin{proof}Let $A$ be the connected component of ${\rm Ker}(1-\alpha^\ast:J_n\lra J_n)$. 
The inclusion $A\subset J_n$ and the map ${\rm Jac}(\tee_n)\lra J_n$ explained before induces an identification 
$$H^0(A,\Omega^1)=H^0(J_n,\Omega^1)^{\alpha^\ast=1}=H^0(\txx_n,\Omega^1)^{\gamma^\ast=-1}=H^0(\tee_n,\Omega^1)$$
which implies that $A\stackrel{k}{\sim}{\rm Jac}(\tee_n)$. 

Put $B=J_n/A$. Since $\beta\alpha \beta=\alpha^{-1}$, the pullback $\beta^\ast$ acts on $B$. 
By Theorem \ref{dif} one can see that 
$$H^0(B,\Omega^1)=H^0(\C_n,\Omega^1)^{\alpha^\ast\not=1,\gamma^\ast=-1}
=\<\widetilde{\omega}_{r,s}, \widetilde{\theta}_i\ |\ 
\tiny{\begin{array}{c}
0\le r,s,r+s\le n-4,r\not=s\\
 2\le i \le n-2 
\end{array}}
\>$$
where $\alpha^\ast$ acts by 
$$
\alpha^\ast(\widetilde{\omega}_{r,s})=\zeta^{r-s}\widetilde{\omega}_{r,s},\ 
\alpha^\ast(\widetilde{\theta}_i)=\zeta^i \widetilde{\theta}_i.
$$
By using this basis  when $n$ is odd one can check that the number of eigenvectors for $\beta^\ast$ with the eigenvalue $+1$ 
is same as those with the eigenvalue $-1$.    
Hence 
the connected component $B_{\pm}$ of the kernel of the map $1\pm \beta^\ast:B\lra B$ is an abelian variety of dimension 
$\frac{1}{2}{\rm dim}B$ respectively 
and we have $B_+\times B_-\stackrel{k}{\sim} B$. The inclusion $B_-\subset B$ and the composition ${\rm Prym}(\tyy_n/\yy_n)\lra J_n\lra B$ induce 
an isogeny $B_-\stackrel{k}{\sim}{\rm Prym}(\tyy_n/\yy_n)$ since 
$H^0(B_-,\Omega^1)=H^0(B,\Omega^1)^{\beta^\ast=1}=H^0({\rm Prym}(\tyy_n/\yy_n),\Omega^1)$ by Proposition \ref{di}-(1). 
By looking the action on $H^0(B_-,\Omega^1)$ we see that $\alpha$ never preserves $B_-$ and this fact gives rise to an isogeny 
$B_-\stackrel{k}{\sim} B_+$.  Since $\alpha^\ast+(\alpha^{-1})^\ast$ commutes with $\beta^\ast$, it acts on $B_{\pm}$. Hence, 
$${\rm End}_k({\rm Prym}(\tyy_n/\yy_n))\otimes_\Z\Q={\rm End}_k(B_\pm)\otimes_\Z\Q\supset \Q(\zeta_n+\zeta^{-1}_n).$$
This proves the first claim. 

For the second claim we consider 
the connected component $B'_{\pm}$ of the kernel of the map $1\pm (\beta\alpha^{\frac{n}{m}})^\ast:B\lra B$. Then natural maps 
$B'_-\subset B\longleftarrow J_n \longleftarrow {\rm Prym}(\tzz_n/\zz_n)$ imply an isogeny $B'_-\stackrel{k}{\sim}{\rm Prym}(\tzz_n/\zz_n)$ since 
$$H^0(B'_-,\Omega^1)=H^0(B,\Omega^1)^{(\beta\alpha^{\frac{n}{m}})^\ast=1}=H^0({\rm Prym}(\tzz_n/\zz_n),\Omega^1).$$
Since $$1-\beta^\ast=1-(\beta \alpha^\frac{n}{m})=(1-(\beta\alpha^{\frac{n}{m}})^\ast)\sum_{j=0}^{m-1}((\beta\alpha^{\frac{n}{m}})^j)^\ast$$, 

$B'_+\subset B\longleftarrow J_n \stackrel{\alpha^\ast}{\longleftarrow} {\rm Prym}(\tyy_n/\yy_n)$ imply an isogeny $B'_+\stackrel{k}{\sim}{\rm Prym}(\tyy_n/\yy_n)$ since 
$$H^0(B'_+,\Omega^1)=H^0(B,\Omega^1)^{(\beta\alpha^{\frac{n}{m}})^\ast=-1}=H^0({\rm Prym}(\tyy_n/\yy_n),\Omega^1).$$
We can easily see that $H^0(B,\Omega^1)^{\beta^\ast=1}$ is a complement of $H^0(B'_+,\Omega^1)$ in $H^0(B,\Omega^1)$ and 
this implies that ${\rm Prym}(\tyy_n/\yy_n)$ is also a complement of $B_-$ in $B$ under the isogeny induced by the previous argument for odd $n$. 
Hence we have 
$$J_n\stackrel{k}{\sim}B_+\times B_-\stackrel{k}{\sim}{\rm Prym}(\tyy_n/\yy_n)\times {\rm Prym}(\tzz_n/\zz_n).$$
It follows from the same reason as above that 
$${\rm End}_k({\rm Prym}(\tzz_n/\zz_n))\otimes_\Z\Q={\rm End}_k(B_+)\otimes_\Z\Q\supset \Q(\zeta_n+\zeta^{-1}_n).$$
The last claim is just a direct computation and so omitted. This completes the proof.  
\end{proof}

\section{Explicit computations for $n=4,5,6,7,8$}

In this section we explicitly decompose $J_n$ for $4\le n\le  8$.  
Most of the computations are straight forward and we skip some of the details. 
We start with the case of $n=5$ and put the case of $n=4$ at the last part. 

\subsection{$C^0_5$ as the first interesting case}

In the case when $n=5$ this curve appears in previous study on a decomposition of the intermediate Jacobians for some cubic threefolds \cite{GY} but 
it was not studied because of the singularities.  
Put $a=a_2$. Recall the  defining equation:
\begin{equation} 
C^0_5:   \quad x^5+z^5+3y^5-5xy^3z+a(x^2yz^2-2xy^3z+y^5) = 0
\end{equation}
By Theorem \ref{main} we see that 
\[ J_5\stackrel{k}{\sim} E\times B^2,\]
where $E$ is an elliptic curve over $k$ and $B$ is an abelian surface over $k$ such that  ${\rm End}_k(B)\otimes_\Z\Q$ contains 
$\Q(\sqrt{5})$. 
In what follows we will compute the equations of $E$ and a hyperelliptic curve $C$ whose Jacobian is 
isogenous over $k$ to $B$. The strategy is similar to \cite{GY}.   
For each curve appears here we work on a smooth model instead of the singular model.  By using the automorphisms defined in (\ref{auto2}) we have the commutative diagram 

\[
\begin{CD}
\C_5  @> 5:1 >>  \C_5/\alpha
@>2:1>>E:= \C_5/\langle \alpha,\delta\gamma \rangle,  \\
  @V 2:1 VV @V 2:1 VV @.\\
 C_5  @>5:1 >> C_5/\alpha @. 
\end{CD}\ \qquad  
\begin{CD}
6  @>  >>  2
@>>> 1  \\
  @V VV @V VV @.\\
 1  @> >> 1 @. \\  
\end{CD} \\
\]
where the right diagram indicates the genus of each curve on the left side. 
The  vertical arrows on the left hand side  stand for the quotient map by $\gamma$ and 
the upper right arrow to $E$ means the quotient map by the group $\alpha,\delta\gamma$ where $\delta$ is 
the hyperelliptic involution of $ \C_5/\alpha$. As we will see below 
the curve $\C_5/\alpha$ turns out to be hyperelliptic. 

Explicit equations of the curves in the diagram are obtained step by step.  For simplicity we let $y=1$.  Then 
\[ s:=x^5  \quad \textit{and} \quad  t:=xz\]  
are fixed by $\alpha$ and this gives rise to 
\begin{equation}
s^2+(a(t-1)^2-5t+3)s+t^5=0.
\end{equation}
Substituting $s$ by 
\[ u=(1-t)s/2+(a(t-1)^2-5t+3)/2,\]
  we obtain 
\[C_5/\alpha : \quad u^2=f^0_5(t),\]
where
\[ f^0_5(t):=-t^3+\frac{1}{4}(a^2-8)t^2- \frac{1}{2}(a^2+5a+6)t+\frac{1}{4}(a+3)^2.\]
Since the vertical arrow in the above diagram is given by $w^2=1-xz=1-t$, one has  
\[\C_5/\alpha:   \quad u^2=f^0_5(1-w^2)     \]
with the hyperelliptic involution $\delta:(s,w)\mapsto (-s,w)$. 
Then clearly 
\[ \delta\gamma:(u,w)\mapsto (-u,-w)\]
 and this gives rise to the smooth model $E$ as below: 
\[   E: \quad y^2=-\frac{1}{4}(4a+15)x^3+\frac{5}{2}(a+4)x^2+\frac{1}{4}(a^2-20)x+1,\]
where $(x,y)=   \left(\ds\frac{1}{w^2},\frac{u}{w^3} \right)$.
By construction we see that 
\begin{prop}
The elliptic curve $E$ is a $k$-simple factor of $J_5$. 
\end{prop}
Similarly we consider the following commutative diagram but this time we use $\beta$ instead of $\alpha$: 
\[
\begin{CD}
\C_5  @> 2:1 >>  \C_5/\beta,  \\
  @V 2:1 VV @V 2:1 VV @.\\
 C_5  @>2:1 >> C_5/\beta @. 
\end{CD}\hspace{5mm}  
\begin{CD}
6  @>  >>  2  \\
  @V VV @V VV @.\\
 1  @> >> 0. @. 
\end{CD}
\]
The genus of $\C_5/\beta$ (resp. $C_5/\beta$) can be computed by counting fixed points of $\beta$ on  $\C_5$ (resp. $C_5$). 
Clearly $T_1=x+z,\ T_2=xz$ are invariant under $\beta$ and these are generators of the function field $k(C_5/\beta)$. 
Then we have an affine model of $C_5/\beta$: 
\[3 + a + T^5_1 - 5 T_2 - 2 a T_2 - 5 T^3_1 T_2 + a T^2_2 + 5 T_1 T^2_2=0.\]
Since it is rational, we would find a rational parameter. 
Put 
\begin{equation}
\label{t1}
t_1=\frac{-5-2 a-5 T^3_1+2 a T_2+10 T_1T_2}{-1+T_1+T^2_1}.
\end{equation}

Then one has 
\[-25 - 8 a + t^2_1 + 10 T_1 + 4 a T_1 - 5 T^2_1=0\]
and it has a rational solution $(t_1,T_1)=(a+5,a+2)$. 
Putting $x=\ds\frac{t_1-a-5}{T_1-a-2}$ and substituting this into the above equation we obtain 
\begin{equation}\label{para}
(t_1,T_1)= \left(\frac{25 + 5 a - 10 x - 6 a x + 5 x^2 + a x^2}{5-x^2},\frac{a - 10 x - 2 a x + 2 x^2 + a x^2}{-5+x^2} \right).
\end{equation}

To gain the equation of $\C_5/\beta$ we substitute $w^2=1-xz=1-T_2$ and 
(\ref{para}) into (\ref{t1}). Simplifying  the factors we obtain 
\[ w^2=-\frac{(-10 - a + a x)g(x)}{(x^2-5)^2(-25 - 5 a + 5 x + 3 a x)},\quad g(x)=\sum_{i=0}^4c_ix^i, \]
where 
\[
\begin{split}
c_4 & =5 + 5 a + a^2, \\
c_3 & =-2 (5 + a) (5 + 2 a), \\
 c_2 & =2 (50 + 20 a + 3 a^2), \\
 c_1 & =-2 (5 + a) (-5 + 2 a), \\
 c_0 & =-25 - 5 a + a^2. 
\end{split}
\] 
We introduce a new parameter 
\[ y=(-25 - 5 a + 5 x + 3 a x)(x^2-5)w \]
 and then we have a birational model of 
 $\C_5/\beta$:
\[C:y^2=h(x),\ h(x)=-(3 a x+ 5 x-25 - 5 a ) (ax- a -10) g(x).\]
The discriminant of $h(x)$ is 
\[ \Delta_h = 2^{30}5^5(15 + 4 a)^2 (-25 - 5 a + a^2)^{14}.\] 

Summing up by Theorem \ref{main} we have.

\begin{prop}\label{n5}
Let $B={\rm Jac}(C)$. Then 
$$J_5\stackrel{k}{\sim}E\times B^2$$
and ${\rm End}_k(B)\otimes_\Z\Q$ contains $\Q(\sqrt{5})$. 
\end{prop}

%
%


\subsection{The case $n=6$}  The genus 4 curve $C^0_6$ has equation  as follows
\begin{eqnarray}\label{C_6^0}
{x}^{6}+{z}^{6}+4\,{y}^{6}-6\,xz{y}^{4}+a_{{2}} \left( {x}^{2}{z}^{2}{y}^{2}-2\,xz{y}^{4}+{y}^{6} \right) \\
+a_{{3}} \left( {x}^{3}{z}^{3}-3\,xz{y}^{4}+2\,{y}^{6} \right) =0 \nonumber
\end{eqnarray}
As in the previous case we work on a smooth model for each curve. 
By using the automorphisms defined in (\ref{auto2}) we have the commutative diagram 
\[
\begin{CD}
\C_6  @> 6:1 >>  \widetilde{\X}_6
@>2:1>>\widetilde{\mathcal{E}}_6:= \C_6/\langle \alpha,\delta\gamma \rangle,  \\
  @V 2:1 VV @V 2:1 VV @.\\
 C_6  @>6:1 >> \X_6 @. 
\end{CD}\ \qquad  
\begin{CD}
  13 @>  >>  3@>>> 2  \\
     @V VV @V VV @.\\
 4  @> >> 1 @. \\  
\end{CD} \\
\]
By Proposition \ref{x}  we have an equation defining $\X_6$:  
\[ \X_6: \quad  u^2= f_6^0 (t)   \]
where $f_6^0 (t)$ has degree 4 and is given as follows  %
\begin{eqnarray}
f^0_6(t)&=& 
\frac{1}{4}(-2+a_3) (2+a_3)t^4+\frac{1}{2}(-4+a_2 a_3+a^2_3)t^3 \nonumber \\ 
&& +\frac{1}{4}(-12+a^2_2-12 a_3-3a^2_3)t^2-\frac{1}{2} (2+a_2+a_3)(4+a_2+2 a_3)t \\
&&+\frac{1}{4}(4+a_2+2 a_3)^2 \nonumber
\end{eqnarray}

Since the vertical arrow in the above diagram is given by $w^2=1-xz=1-t$, one has the genus 3 hyperelliptic curve   
\[\widetilde\X_6 :   \quad u^2=f^0_6(1-w^2)     \]
with the hyperelliptic involution $\delta:(u, w)\mapsto (-u, w)$.   Then clearly 
\[ \delta \, \gamma:   (u, w)\mapsto (-u,-w)\]
 and this gives rise to an affine smooth model of $\widetilde{\mathcal{E}}_6$ as below:
\begin{eqnarray}
\widetilde{\mathcal{E}}_6:y^2&=& 
\frac{1}{4} (-2 + a_3) (2 + a_3)x^5+\frac{1}{2} (12 - a_2 a_3 - 3 a_3^2)x^4\nonumber  \\
&&+\frac{1}{4}(-60 + a^2_2 - 12 a_3 + 6 a_2 a_3 + 9 a^2_3)x^3+(20 + 3 a_2 + 10 a_3)x^2\\
&&+(-6 - a_2 - 3 a_3)x \nonumber
\end{eqnarray}   
where $(x,y)=   \left(w^2,uw \right)$.

Next we determine the  equations for the quotient curves $\Y_6$ and $\widetilde\Y_6$ and their genii.  
Using the automorphisms defined in (\ref{auto2}) we have the following diagram
\[
\begin{CD}
\C_6  @> 2:1 >>  \widetilde\Y_6=\C_6/\beta,  \\
  @V 2:1 VV @V 2:1 VV @.\\
 C_6  @>2:1 >> \Y_6=C_6/\beta @. 
\end{CD}\hspace{5mm}  
\begin{CD}
13  @>  >>  6  \\
  @V VV @V VV @.\\
 4  @> >> 2. @. 
\end{CD}
\]
By Proposition \ref{y} 
an defining equation of $\Y_6$ is given by 
\begin{eqnarray}
 g_6(u_,v)&:=&{u}^{6}-6\,{u}^{4}v+9\,{u}^{2}{v}^{2}+a_3{v}^{3}-2\,{v}^{3}+{v}^{2
}a_{{2}}-2\,va_{{2}}\\
&&-3va_{{3}}-6\,v+a_{{2}}+2\,a_{{3}}+4
=0  \nonumber
\end{eqnarray}
Recall that this is of genus 2. We now try to get a Weierstrass model of  $Y_6$ as follows. 
Observe that it has a singularity at $(u,v)=(1,1)$ and introduce $(u_1,v_1)=(u-1,v-1)$. 
By blowing at $(u_1,v_1)=(0,0)$ which corresponds to considering $v_1=tu_1$ with a new parameter $t$. 
Then we have an equation $f(u_1,t)=0$ which is degree 4 (resp. 3) in $u_1$ (resp. $t$). This 
equation has a singularity at $(u_1,t)=(-2,0)$. We substitute $u_2=u_1+2$ and take the blowing up $u_2=u_3t$. Then we have an 
equation 
$$Q_2(u_3)t^2+Q_1(u_3)t=Q_0(u_3),\ Q_i\in k[u_3],\ {\rm deg}_{u_3}Q_i=2+i\ {\rm for}\  i=0,1,2$$
which gives rise to a defining equation of $\Y_6$. Substituting $t$ with $\ds\frac{t}{Q_2(u_3)}$ we will get 
a hyperelliptic model. To be more precise, put $x=\ds\frac{u^2-1}{v-1}$ and 
$$y=\frac{u (-2- a_3 + 3 u^4 + u^6 + (6 +3a_3 - 6 u^2  - 6 u^4)v+  (- 3- 3a_3  + 9 u^2 )v^2 - 
   (2 + a_3) v^3)}{(v-1)^3}.$$
Then we have an affine smooth model:
$$\Y_6:y^2=(x^3- 6 x^2 +9x+a_3-2) (4 x^3 - 12 x^2 +6 x  - (a_2+3a_3) x +a_3-2).$$   
Next we consider a singular model of $\widetilde\Y_6$ 
$$g_6(u,1-w^2)=0.$$
This curve has two automorphisms 
$$\tau_1:(u,w)\mapsto (-u,w),\ \tau_2:(u,w)\mapsto (-u,-w).$$ 
Let us compute the equations of two curves $\widetilde\Y_6/\<\tau_i\>,\ i=1,2$ which turn out to be factors of Prym$(\widetilde\Y_6/\Y_6)$. 
Consider the equation $g^1_6(u',w)=0$ so that $g^1_6(u^2,w)=g_6(u,w)$ (hence we putted $u'=u^2$). Then,  
the curve defined by $g^1_6(u',w)=0$ has a singularity at $(u',w)=(1,0)$ in the $(u',w)$-plane. Substitute $u_1=u'-1$ and set $u_1=u_2w^2$. 
Then we have the 
equation 
$$Q^1_3(u_2)w^2=Q^1_2(u_2),\ Q_i\in k[u_2],\ {\rm deg}_{u_2}Q^1_i=i\ {\rm for}\  i=2,3.$$ 
Substituting $t$ with $\ds\frac{t}{Q^1_3(u_2)}$ we will have 
a hyperelliptic model. To be more precise, put  
Put $$x=\ds\frac{1-u'}{w^2},\ y=\frac{(2-a_3)w^6+9(u'-1)w^4+6(u'-1)^2w^2+(u'-1)^3}{w^5}.$$
This gives rise to an affine smooth model of $\widetilde\Y_6/\<\tau_1\>$:
$$\widetilde\Y_6/\<\tau_1\>:y^2=(3x^2-6x-3-a_2-3a_3)(x^3-6x^2+9x-2+a_3).$$ 

For the curve $\widetilde\Y_6/\<\tau_2\>$ we consider two functions $s=u w$ and $t=u^2$ which is invariant under $\tau_2$. 
Consider the polynomial $g^2_6(s,t)$ so that $g^2_6(uw,u^2)=g_6(u,w)$. This curve has singularities at $(s,t)=(0,0),(0,1)$. 
Therefore put $s^2=Xt(t-1)$. Then,  after deleting the factor $(t-1)^2$ the polynomial $g^2_6(s,t)=0$ implies   
$$4 - t + (12  - 6 t) X +( 6 - 9 t    -
    a_2 - 3a_3) X^2 +(2- a_3- 2 t + a_3 t )X^3=0.$$
Solving this in $t$, we have 
$$t=\frac{-4 - 12 X  +(- 6 + a_2 + 3a_3) X^2 +(-2+a_3) X^3}{-1 - 
 6 X - 9 X^2 +(- 2 + a_3)X^3}.$$ 
Therefore we have 
$$s^2=Xt(t-1)=\frac{X (-3 - 6 X + (3 + a_2+3a_3)X^2)}{-1 - 
 6 X - 9 X^2 +(- 2 + a_3)X^3}.$$
Put $Y=s(-1 - 
 6 X - 9 X^2 +(- 2 + a_3)X^3)$ and then we have a hyperelliptic model 
$$\widetilde\Y_6/\<\tau_2\>:Y^2=X \{(3 + a_2+3a_3)X^2 - 6 X-3 \}\{(- 2 + a_3)X^3 - 9 X^2-1-6 X-1\}.$$ 
Next, is left  to analyze $\tzz_6$ and $\zz_6$. For these curves  we have the following diagram
\[
\begin{CD}
\C_6  @> 2:1 >>  \tzz_6=\C_6/\<\alpha^3\beta\>,  \\
  @V 2:1 VV @V 2:1 VV @.\\
 C_6  @>2:1 >> \zz_6=C_6/\<\alpha^3\beta\> @. 
\end{CD}\hspace{5mm}  
\begin{CD}
13  @>  >>  4  \\
  @V VV @V VV @.\\
 4  @> >> 1. @. 
\end{CD}
\]

Notice that the curve $\tzz$ inherits two involutions 
$$\sigma_1:(u,w)\mapsto (-u,w),\ \sigma_2:(u,w)\mapsto (-u,-w).$$

By Proposition \ref{z} 
an defining equation of $\zz_6$ is given by 
\begin{eqnarray}
 h_6(u_,v)&:=&u^6  + 6 u^4 v + 9 u^2 v^2 + 2 v^3- 6 v+4\\
&&+{v}^{2}a_{{2}}-2\,va_{{2}}-3\,va_{{3}}-6\,v+a_{{2}}+2\,a_{{3}}+4 \nonumber
=0.  
\end{eqnarray}
Since it is of genus 1 we try to find an Weierstrass model for $\zz_6$ as follows. 
Consider the equation $h^1_6(u_1,v)$ so that $h^1_6(u^2,v)=h_6(u,v)$. 
Then the curve defined by $h^1_6(u_1,v)=0$ is rational and it has the singularities 
at $(u_1,v)=(0,1),(-3,1)$ on $(u_1,v)$-plane. Put $v=v_1+1$ and $v_1=v_2u_1(u_1+3)$. 
Further we put $u_1=\ds\frac{t}{v_2}$ to remove the remaining singularities. Then we have 
$$v_2=\frac{-1 - 6 t - 9 t^2 - 2 t^3 - a_3 t^3}{t (6 + 6 t + a_2 + 3a_3 + 3a_3 t)}$$ 
and this gives 
$$u^2=u_1=\ds\frac{t}{v_2}=\frac{t^2 (6 + 6 t + a_2 + 3a_3 + 3a_3 t)}{-1 - 6 t - 9 t^2 - 2 t^3 - a_3 t^3}.$$
Put $s=\ds\frac{-1 - 6 t - 9 t^2 - 2 t^3 - a_3 t^3}{t}u$. Then we have a smooth affine Weierstrass model
$$\zz_6:s^2=-\{ (3a_3+6)t + a_2 + 3a_3+6 \}\{(a_3+2) t^3 + 9 t^2+6t+1 \}.$$
On the other hand 
$$w^2=1-v=-t\Big(2+\frac{t}{v_2}\Big)=\frac{t (-3 - 18 t +(- 21 + a_2+ 3a_3) t^2)}{1 + 6 t + 9 t^2 + 
 (2+a_3)t^3}.$$
 Put $w'=(1 + 6 t + 9 t^2 + 
 (2+a_3)t^3)w$ and then we have an affine smooth model of the quotient curve $\tzz_6/\<\sigma_1\>$: 
 $$\tzz_6/\<\sigma_1\>: w'^2=t (-3 - 18 t +(- 21 + a_2+ 3a_3) t^2)(1 + 6 t + 9 t^2 + 
 (2+a_3)t^3).$$
The computation for $\tzz_6/\<\sigma_2\>$ will be similar to what we carried out for $\widetilde\Y_6\<\tau_2\>$. 
Put $X-uw$ and $Y-w^2$. Then we have a singular model
\begin{eqnarray}\label{h26}
h^2_6(X,Y)&=&X^6 + 6 X^4 Y + 9 X^2 Y^2 - 6 X^4 Y^2 - 18 X^2 Y^3 + 9 X^2 Y^4\\
&& + 
 6 Y^5 - 2 Y^6 +a_2 Y^5  + 3a_3 Y^5  -a_3 Y^6=0.\nonumber
 \end{eqnarray}
 Put $x_1=X^2$ and $x_1=Y^2 x_2$ to remove the singularities. 
 Further we put 
 $$Y=\frac{y_1}{-2 + 9 x_2 - 6 x^2_2 + x^3_2 - a_3},\ y_1=y_2 - \frac{1}{2} (6 - 18 x_2 + 6 x^2_2 + a_2 + 3 a_3).$$
 The equation (\ref{h26}) gives us that a quadratic equation which has a solution $(x_2,y_2)=(0,\frac{1}{2} (6 + a_2 + 3 a_3))$. 
 Put $x_2=t(y_2-\frac{1}{2} (6 + a_2 + 3 a_3))$. Then we have 
 $$X^2=x_2=\frac{tq(t)^3(-1 + 18 t^2 + 3a_2 t^2  + 9a_3 t^2 )}{p(t)^2}$$
 where $q(t)=6 + 36 t + a_2 + 9a_2 t  + 3a_3 + 18a_3 t$ and $p(t)$ is a polynomial of (generic) degree 3 in $t$. 
We substitute $y=X\ds\frac{p(t)}{q(t)}$ and then obtain  
$$\tzz_6/\<\sigma_2\>:y^2=t\{ (36   + 9a_2   + 18a_3) t + 3a_3+ a_2+6\}\{( 18  + 3a_2   + 9a_3 )t^2 -1\}.$$  
Summing up we have the following result. 

\begin{prop}\label{j6}Keep the notation as above. Then 
$$J_6\stackrel{k}{\sim}{\rm Jac}(\widetilde{\mathcal{E}}_6)\times {\rm Jac}(\widetilde\Y_6/\<\tau_1\>)
\times {\rm Jac}(\widetilde\Y_6/\<\tau_2\>)
\times {\rm Jac}(\tzz_6/\<\sigma_1\>)
\times {\rm Jac}(\tzz_6/\<\sigma_2\>).$$
\end{prop}
\begin{proof}By Theorem \ref{main} we have 
$$J_6\stackrel{k}{\sim}{\rm Jac}(\widetilde{\mathcal{E}}_6)\times {\rm Prym}(\tyy_6/\yy_6)\times {\rm Prym}(\tzz_6/\zz_6).$$
Applying Corollary \ref{easy} to $H=\<\tau_1,\tau_2\>$ for $\tyy_6$ we have 
$${\rm Jac}(\tyy_6)\stackrel{k}{\sim} {\rm Jac}(\yy_6)\times {\rm Jac}(\widetilde\Y_6/\<\tau_1\>)
\times {\rm Jac}(\widetilde\Y_6/\<\tau_2\>)$$
since the genus of $\tyy_6/H$ is zero. This implies ${\rm Prym}(\tyy_6/\yy_6)\stackrel{k}{\sim} {\rm Jac}(\widetilde\Y_6/\<\tau_1\>)
\times {\rm Jac}(\widetilde\Y_6/\<\tau_2\>)$. We have 
${\rm Prym}(\tzz_6/\zz_6)\stackrel{k}{\sim}  {\rm Jac}(\tzz_6/\<\sigma_1\>)
\times {\rm Jac}(\tzz_6/\<\sigma_2\>)$ as well. 
\end{proof}


\subsection{The case $n=7$}   The genus 8 curve $C_7$ has equation as follows
\begin{eqnarray}\label{C7}
{x}^{7}+{z}^{7}+5\,{y}^{7}-7\,xz{y}^{5}+a_{{2}} \left( {x}^{2}{z}^{2}{y}^{3}-2\,xz{y}^{5}+{y}^{7} \right) +\\
a_{{3}} \left( {x}^{3}{z}^{3}y-3\,xz{y}^{5}+2\,{y}^{7} \right) =0 \nonumber
\end{eqnarray}
By Proposition \ref{x} we have the following  diagram
\[
\begin{CD}
\C_7  @> 7:1 >>  \widetilde\X_7=\C_7/\alpha
@>2:1>> \widetilde{\mathcal{E}}_7=\C_7/\langle \alpha,\delta\gamma \rangle,  \\
  @V 2:1 VV @V 2:1 VV @.\\
 C_7  @>7:1 >> \X_7=C_7/\alpha @. 
\end{CD}\ \qquad  
\begin{CD}
  22 @>  >>  4
@>>> 2  \\
  @V VV @V VV @.\\
 8  @> >> 2 @. \\  
\end{CD} \\
\]

By Proposition \ref{x} we have 

\begin{eqnarray}
f^0_7(t)&=&-t^5+\frac{1}{4} (-8 + a^2_3)t^4+\frac{1}{2}(-6 + a_2 a_3 + a^2_3)t^3 \nonumber \\
&&+\frac{1}{4} (-16 + a^2_2 - 14 a_3 - 3 a^2_3)t^2-+\frac{1}{2}(2 + a_2 + a_3) (5 + a_2 + 2 a_3)t  \\
&&+\frac{1}{4}(5 + a_2 + 2 a_3)^2 \nonumber 
\end{eqnarray}

Since the vertical arrow in the above diagram is given by $w^2=1-xz=1-t$, one has  
\[\widetilde\X_7 :   \quad u^2=f^0_7(1-w^2)     \]
with the hyperelliptic involution $\delta:(u, w)\mapsto (-u, w)$.   Then clearly 
\[ \delta \, \gamma:   (u, w)\mapsto (-u,-w)\]
 and this gives rise to a smooth model of $ \widetilde {\mathcal E}_7 $ as below:
\begin{eqnarray}
\widetilde {\mathcal E}_7:y^2&=& \Big(-7 + \frac{a^2_3}{4} \Big)x^5+\Big(21 - \frac{1}{2} a_2 a_3 
+- \frac{3 a^2_3}{2}\Big)x^4  \nonumber \\
&&+\Big(-35 + \frac{a^2_2}{4} - \frac{7 a_3}{2} + \frac{3}{2} a_2 a_3 + \frac{9 a^2_3}{4}\Big)x^3 \\
&&+\Big(35 + \frac{7 a_2}{2} + \frac{23 a_3}{2}\Big)x^2+
\Big(-\frac{35}{4} - a_2 - 3 a_3\Big)x   \nonumber 
\end{eqnarray}
where $(x,y)=   \left(w^2,uw \right)$ and it is of genus $g= 2$.

Next we determine equations for the quotient curves $\Y_7$ and $\widetilde\Y_7$ and their genii.  
Start with a singular model of $C_7$ which has the equation given as in (\ref{C7}).  By using the automorphisms we have the following diagram 
\[
\begin{CD}
\C_7  @> 2:1 >>  \C_7/\beta,  \\
  @V 2:1 VV @V 2:1 VV @.\\
 C_7  @>2:1 >> C_7/\beta @. 
\end{CD}\hspace{5mm}  
\begin{CD}
22  @>  >>  9  \\
  @V VV @V VV @.\\
 8  @> >> 3. @. 
\end{CD}
\]

By Proposition \ref{y}  the curve $\Y_7=C_7 / \<\beta \>$ is of genus  3 and it is defined by 
\begin{eqnarray}
g_7(u,v)&:=&{u}^{7}-7\,{u}^{5}v+14\,{u}^{3}{v}^{2}-7\,u{v}^{3}+{v}^{3}a_{{3}}+{v}^
{2}a_{{2}}\\
&&-2\,va_{{2}}-3\,va_{{3}}-7\,v+a_{{2}}+2\,a_{{3}}+5=0. \nonumber  
\end{eqnarray}
The curve $\widetilde\Y_7$ is given by $g_7(u,1-w^2)=0$.  Summing up we have proved the following.

\begin{prop}Keep the notation as above. Then, 
$$J_7\stackrel{k}{\sim}{\rm Jac}(\widetilde{\mathcal{E}}_7)\times J_{7,\beta}^2 $$
where ${\rm End}(J_{7,\beta})\otimes_\Z\Q\supset \Q(\zeta_7+\zeta^{-1}_7)$. 
\end{prop}

It would be interesting to find a Jacobian which is isogenous to $J_{7,\beta}$. We expect there exists 
a smooth curve $D$ of genus 3 (maybe non-hyperelliptic) so that ${\rm Jac}(D)^2\stackrel{\overline{k}}{\sim} J_{7,\beta}$ with 
${\rm End}_{\overline{k}}({\rm Jac}(D))\otimes_\Z\Q\supset \Q(\zeta_7+\zeta^{-1}_7)$. 



\subsection{The case $n=8$} 
The genus 13 curve $C_8$ has the following equation: 
\begin{eqnarray}\label{C_7}
{x}^{8}+{z}^{8}+6-8\,xz+a_{{2}} \left( {z}^{2}{x}^{2}-2\,xz+1 \right) 
+a_{{3}} \left( {z}^{3}{x}^{3}-3\,xz+2 \right)\\ 
+a_{{4}} \left( {z}^{4}
{x}^{4}-4\,xz+3 \right) 
=0.  \nonumber
\end{eqnarray}
By Proposition \ref{x} we have the following  diagram
\[
\begin{CD}
\C_8  @> 8:1 >> \txx_8=\C_8/\alpha
@>2:1>> \tee_8=\C_8/\langle \alpha,\delta\gamma \rangle,  \\
  @V 2:1 VV @V 2:1 VV @.\\
 C_8  @>8:1 >> \xx_8=C_8/\alpha @. 
\end{CD}\ \qquad  
\begin{CD}
33 @> >> 5  
@>>> 3  \\
  @V VV @V VV @.\\
13  @> >> 2 @. \\  
\end{CD} \\
\]

By Proposition \ref{x} we have an affine smooth model 
\[ \X_8: \quad  u^2= f_8^0 (t)   \]
where $f_8^0 (t)$ has degree 6 and is given as follows  %
\[ 
\begin{split}
f_8^0(t)= &\frac{1}{4}\, \left( a_{{4}}{t}^{2}+2\,{t}^{2}+a_{{3}}t+2\,a_{{4}}t+4\,t+a_{{2
}}+2\,a_{{3}}+3\,a_{{4}}+6 \right)\\
&  \left( a_{{4}}{t}^{4}-2\,{t}^{4}+a
_{{3}}{t}^{3}+a_{{2}}{t}^{2}-2\,a_{{2}}t-3\,a_{{3}}t-4\,a_{{4}}t-8\,t+
a_{{2}}+2\,a_{{3}}+3\,a_{{4}}+6 \right) 
\end{split} \]
By Proposition \ref{x} one has  
\[\widetilde\X_8 :   \quad u^2=f^0_8(1-w^2)     \]
with the hyperelliptic involution $\delta:(u, w)\mapsto (-u, w)$ and it is of genus 5. The involution  
\[ \delta \, \gamma:   (u, w)\mapsto (-u,-w)\]
 gives rise to a smooth model $ \widetilde {\mathcal E}_8$ as below: 
$$\tee_8:y^2=xf^0_8(1-x)$$
where $(x,y)=   \left(w^2,wu \right)$ and it is of genus 3.

Next we determine the equations for the quotient curves $\Y_8, \widetilde\Y_8$, $\ww_8$, and $\tww_8$. 
By Proposition \ref{y} and \ref{w} we have the following diagram: 
\[
\begin{CD}
\tzz_8 @< 1:2 <<  \C_8  @> 2:1 >> \tyy_8=\C_8/\beta,  \\
 @V 2:1 VV   @V 2:1 VV @V 2:1 VV @.\\
\zz_8 @< 1:2 << C_8  @>2:1 >> \yy_8=C_8/\beta @. 
\end{CD}\hspace{5mm}  
\begin{CD}
13 @<  <<33  @>  >>  15  \\
  @V  VV @V VV @V VV @.\\
5 @<  << 13  @> >> 6. @. 
\end{CD}
\]

By Proposition \ref{y}  the curve $\Y_8=C_8 / \<\beta \>$ is of genus 6 and it is defined by 
\begin{eqnarray}
g_8(u,v)&:=& u^8 - 8 v - 8 u^6 v + 20 u^4 v^2 - 16 u^2 v^3 + 2 v^4 \nonumber  \\
&&+6+ a_2 - 
 2a_2 v  + a_2v^2 + 2a_3  - 3a_3 v  \\
&&+a_3 v^3  + 3a_4 - 
 4a_4 v  + a_4 v^4=0  \nonumber  
\end{eqnarray}
where $(u,v)=(x+z,xz)$. 
The curve $\widetilde\Y_8$ is given by $g_8(u,1-w^2)=0$. 
Recall that as in the case of $n=6$, the curve $\tyy_8$ admits two involutions 
$$\tau_1:(u,w)\mapsto (-u,w),\ \tau_2:(u,w)\mapsto (-u,-w).$$
Put $H=\<\tau_1,\tau_2\>$. Then we have the quotient curve $\tzz_8/H=\zz_8/\tau_3$ where $\tau_3:(u,v)\mapsto (-u,v)$. 
It is easy to see that this curve has genus 2. By Corollary \ref{easy} we have 
$${\rm Jac}(\tyy_8)\times {\rm Jac}(\tyy_8/H)^2\stackrel{k}{\sim} {\rm Jac}(\yy_8)\times {\rm Jac}(\tyy_8/\<\tau_1\>)
\times {\rm Jac}(\tyy_8/\<\tau_2\>)$$
which induces 
$${\rm Prym}(\tyy_8/\yy_8)\stackrel{k}{\sim} {\rm Prym}((\tyy_8/\<\tau_1\>)/(\tyy_8/H))\times 
 {\rm Prym}((\tyy_8/\<\tau_2\>)/(\tyy_8/H)).$$

On the other hand by Proposition \ref{z}  the curve $\zz_8=C_8 / \<\beta \alpha \>$ is of genus 5 and it is defined by 
\begin{eqnarray}
 -6 - u^8 + 8 v + 8 \zeta^{-1}_8 u^6 v + 20 \zeta^{2}_8 u^4 v^2 - 
 16 \zeta_8 u^2 v^3 + 2 v^4\nonumber   \\
-a_2(v-1)^2-a_3( v-1)^2 ( v+2)-a_3(v-1)^2 (v^2 + 2 v + 3)=0   \nonumber  
\end{eqnarray}
where $(u,v)=(x+\zeta^{-1}_8 z,xz)$. Replacing $u$ with $\zeta^3_{16}u$ we have a nicer form 
\begin{eqnarray}
h_8(u,v)&:=& -6 + u^8 + 8 v + 8  u^6 v + 20  u^4 v^2 +16 u^2 v^3 + 2 v^4-a_2(v-1)^2 \\
&&-a_3( v-1)^2 ( v+2)-a_3(v-1)^2 (v^2 + 2 v + 3)=0.   \nonumber  
\end{eqnarray}

The curve $\tzz_8$ with genus 13 is given by $h_8(u,1-w^2)=0$. 
Notice that the curve $\tzz_8$ admits two involutions 
$$\kappa_1:(u,w)\mapsto (-u,w),\ \kappa_2:(u,w)\mapsto (-u,-w).$$
Put $H'=\<\kappa_1,\kappa_2\>$. Then we have an elliptic curve $E_8:=\tzz_8/H'$ over $k$. As seen before we have 
$${\rm Prym}(\tzz_8/\zz_8)\stackrel{k}{\sim} {\rm Prym}((\tzz_8/\<\kappa_1\>)/E_8)\times 
 {\rm Prym}((\tzz_8/\<\kappa_2\>)/E_8).$$
 
Summing up we have the following .

\begin{prop}\label{j8}Keep the notation as above. Then,  
$$J_8\stackrel{k}{\sim}{\rm Jac}(\widetilde{\mathcal{E}}_8)\times {\rm Prym}((\tyy_8/\<\tau_1\>)/(\tyy_8/H))\times 
 {\rm Prym}((\tyy_8/\<\tau_2\>)/(\tyy_8/H))$$
$$\times {\rm Prym}((\tzz_8/\<\kappa_1\>)/E_8)\times 
 {\rm Prym}((\tzz_8/\<\kappa_2\>)/E_8)$$
where  the endomorphism rings tensored by $\Q$ for the last four factors contain $\Q(\zeta_8+\zeta^{-1}_8)=\Q(\sqrt{2})$.  
\end{prop}
\begin{proof}The claim follows from Theorem \ref{main} and the computation done above.  
\end{proof}
It would be interesting to find all Jacobian surfaces over $k$ with real multiplication by $\sqrt{2}$ which are  factors  of the above 
Prym varieties.   

\subsection{$C^0_4$}
Finally we treat the case of $n=4$ which is the most simplest case. So we give explicit equations without details. 
Recall the equation defining:
\begin{equation} 
C^0_4:   \quad x^4 + z^4  + 2 y^4 - 4 x y^2 z+ a_2( x^2 z^2 - 2 x y^2 z + y^4)= 0
\end{equation}
By Theorem \ref{main} we see that 
$$J_4\stackrel{k}{\sim} E_1\times E_2\stackrel{k}{\simeq} E^2_1$$
where 
$$E_1:y^2=(a_2+2)x^4+4x^2+1,\ E_2:y^2=(a_2+2)x^4-4x^2+1$$ 
and the isomorphism $E_1\simeq E_2$ is given by $x\mapsto \sqrt{-1}x$. Note that $\sqrt{-1}=\zeta_4\in k$.  



\begin{thebibliography}{PQ}


\bibitem[GY]{GY} B.\ van Geemen and T.\ Yamauchi, On intermediate Jacobians of cubic 
threefolds admitting an automorphism of order five, to appear in the Looijenga volume of the Pure and Applied Math Quarterly. 


\bibitem[Mum]{Mpv} D.\ Mumford, 
{\it Prym varieties I}.  
Contributions to analysis (a collection of papers dedicated to Lipman Bers), 
325--350. Academic Press, New York, 1974.

\bibitem[KR]{KR}  Kani, E.; Rosen, M. Idempotent relations and factors of Jacobians. Math. Ann. 284 (1989), no. 2, 307-327.

\bibitem[BSS]{beshaj} Beshaj, L. and Shaska, T. and Shor, C.  On Jacobians of curves with superelliptic components, Riemann and {K}lein surfaces, automorphisms, symmetries and moduli spaces,    Contemp. Math., (2014), vol  629, pg.   1--14.


\bibitem[BSh]{dec} Beshaj, L. and Shaska, T.  Decomposition of some Jacobian varieties of dimension 3, Artificial Intelligence and Symbolic Computation,   Lecture Notes in Comput. Sci., (2014), vol 8884, pg 193--204. 

\bibitem[Sh]{MR3508311} Shaska, T. Genus two curves with many elliptic subcovers, Comm. Algebra, (2016), vol. 44, nr. 10, pg. 4450--4466. 

\end{thebibliography}

\end{document}